%% file: hyptriangles.tex
\documentclass[12pt]{amsart}

\input{preambles.tex}

\title{Hyperbolic Heron Triangles and Elliptic Curves}

\author{Matilde Lalín}
\address{Université de Montréal,
  Pavillon André-Aisenstadt,
  Département de mathématiques et de statistique,
  CP 6128,
  succ.\ Centre-ville, 
	Montréal, Québec, H3C~3J7, Canada}
  \email{mlalin@dms.umontreal.ca, \textit{\fontfamily{\familydefault}\selectfont Web: }%
  dms.umontreal.ca/\textasciitilde mlalin}
\author{Olivier Mila}
\address{Centre de recherches mathématiques,
	Université de Montréal,
  Pavillon André-Aisenstadt,
  2920 Chemin de la tour,
	Montréal, Québec, H3T~1J4, Canada}
  \email{olivier.mila@umontreal.ca, \textit{\fontfamily{\familydefault}\selectfont Web: }crm.umontreal.ca/\textasciitilde mila}

\begin{document}
\begin{abstract}
    We define hyperbolic Heron triangles (hyperbolic triangles with ``rational'' side-lengths and area)
    and parametrize them in two ways as rational points of certain elliptic curves.
     We show that there are infinitely many hyperbolic Heron 
    triangles with one angle $\alpha$ and area $A$ for any (admissible) choice of $\alpha$ and $A$; 
    in particular, the congruent number problem has always infinitely many solutions in 
    the hyperbolic setting. 
    We also explore the question of hyperbolic triangles with a rational median and 
    a rational area bisector (median splitting the triangle in half).
    %introduce the concept of area bisector (median splitting the triangle in half).
\end{abstract}
\thanks{This work is supported by the Swiss National Science Foundation, Project number \texttt{P2BEP2\_188144}, by the Natural Sciences and Engineering Research Council of Canada, Discovery Grant \texttt{355412-2013}, and by the Fonds de recherche du Qu\'ebec - Nature et technologies, Projet de recherche en \'equipe \texttt{256442}}
\maketitle
\section*{} \label{sec:intro}
\input{intro}
\section{Hyperbolic heron triangles -- Angle parametrization} \label{sec:heron-angles}
\input{heron-angles}

\section{Hyperbolic heron triangles -- Side length parametrization}\label{sec:heron-sides}
\input{heron-sides}

\section{Equilateral Triangles} \label{sec:equilateral}
\input{equilateral}

\section{Rational medians} \label{sec:medians}
\input{medians}

\section{Area bisectors} \label{sec:area-bisectors}
\input{area-bisectors}

%{\small
\nocite{*}
\bibliographystyle{abbrv}
\bibliography{bibli} 
%}

\end{document}

%% file: preambles.tex
\usepackage[utf8]{inputenc}
\usepackage[T1]{fontenc}             

\usepackage{graphicx,%
				amsfonts,%
				amsthm,%
				amsmath,%
				amssymb,%
        geometry,%
        xcolor,%
        mathtools,%
				stackrel,%
				array,%
        xifthen,
        enumitem,
}
\usepackage[square,sort,comma,numbers]{natbib}
\usepackage[hidelinks]{hyperref}

\newcommand{\dmo}[2][]{%
 \expandafter\DeclareMathOperator\csname#2\endcsname%
  {\ifthenelse{\isempty{#1}}%
    {#2}% if #1 is empty
    {#1}% if #1 is not empty#1}%
  }%
}

\theoremstyle{plain}
\newtheorem{thm}{Theorem}[section]
\newtheorem{thmAbs}{Theorem}
\newtheorem{theorem}[thmAbs]{Theorem}
\newtheorem{lemma}[thm]{Lemma}
\newtheorem{corollary}[thmAbs]{Corollary}
\newtheorem{proposition}[thm]{Proposition}
\newtheorem{propositionAbs}[thmAbs]{Proposition}

\theoremstyle{definition}

\newcommand{\R}{\mathbb{R}}
\newcommand{\Z}{\mathbb{Z}}
\newcommand{\Q}{\mathbb{Q}}

\newcommand{\C}{\mathbb{C}}

\dmo{rk}

  \renewcommand{\arraystretch}{1.2}

%% file: intro.tex
The problem of finding triangles with rational area and side lengths in the Euclidean plane goes at least 
as far back as $\sim$ 600 A.D with the Indian mathematician Brahmagupta (see 
\cite{Goins-Maddox}).
If the triangle is assumed right, this is the classical congruent number problem (a number is congruent if 
it is the area of a right triangle with rational sides). 
Remarkably, this problem is equivalent to finding (non-torsion) rational solutions to the elliptic curve 
$y^2 = x^3 - n^2 x$.
For non-right triangles, it was shown by Goins and Maddox \cite{Goins-Maddox} that Heron triangles are in correspondence 
with rational points on the curve $y^2 = x(x- n\tau)(x + n\tau^{-1})$, where $n$ denotes the area and 
$\tau$ is the tangent of half of an angle.

In this paper we investigate the analog problem in the hyperbolic plane.
The first concept we need to transport is that of rationality.
Unlike the Euclidean case where trigonometric laws are polynomial in the area, side lengths and sine and 
cosine of the angles, their hyperbolic counter part involve the \emph{hyperbolic} sine and cosine of the 
side lengths, and the sine and cosine of the area.
For instance, in a triangle with side lengths $a,b,c$ and a right angle opposing side $a$, 
Pythagoras' Theorem and the area $A$ are given by:
\[
    \cosh(a) = \cosh(b)\cosh(c) \text{ and } 
    \sin(A) = \frac{\sinh(b)\sinh(c)}{\cosh(a)+1}
\]
respectively.
It is thus natural to ask that the sine/cosine of the angles be rational, and similarly for the hyperbolic 
functions of the sides, instead of directly asking that these quantities be rational.

For the rationality of the hyperbolic functions on the side lengths, at least two conventions have been 
used in the literature.
One such choice was made by Brody and Schettler~\cite{Brody-Schettler}. 
They call a triangle \emph{rational} if the hyperbolic tangent of its side lengths are rational, and 
they use this definition to prove a correspondence between rational triangles of semi-perimeter 
$\tanh^{-1}(\sigma)$ and inradius $\sinh^{-1}(\rho)$ and rational points on the curve 
$\sigma(x^2 y^2 + xy + \rho^2 (x^2 + xy + y^2 - 1)) = (1 + \rho^2 )(x^2 y + xy^2 )$.

A second (stronger) choice is that of Hartshorne and van Luijk \cite{Hartshorne-vanLuijk}.
Here they call a length $x$ \emph{rational} if $e^x \in \Q$, and then study Pythagorean triples in this context.
This is the notion we will adopt in this paper.

For the area, the Gauss--Bonnet theorem implies that a triangle with angles $\alpha, \beta, \gamma$ has 
area 
\[
A = \pi - \alpha - \beta - \gamma.
\]
The previous discussion together with the above formula for the area suggest a common definition of rationality for angles and area: we
call an area $A$ and an angle $\alpha$ rational if the sine and cosine of these quantities are rational (or 
equivalently if $e^{iA}, e^{i\alpha} \in \Q[i]$).
Thus a hyperbolic triangle with area $A$, angles $\alpha, \beta, \gamma$ and side lengths 
$a,b,c$ is a \emph{hyperbolic Heron triangle} if
\[
    e^a, e^b, e^c \in \Q \quad \text{and} \quad e^{i\alpha}, e^{i\beta}, e^{i\gamma}, e^{iA} \in \Q[i].
\]

By the well known circle parametrization, $e^{ix} \in \Q[i]$ if and only if 
$\cos(x) = \frac{1 - t^2}{1 + t^2}$ and $\sin(x) = \frac{2t}{1+t^2}$ for some $t \in \Q$.
By abuse of terminology we will call $t$ the \emph{rational angle} (resp.\ \emph{rational area}) of a 
hyperbolic triangle if its angle (resp.\ area) is $x$.

Our main result is the following:
\begin{theorem}
    \label{th:heron-angles}
    There is a one-to-one correspondence between hyperbolic Heron triangles having rational area $m$ 
    and one rational angle $u$, and rational points satisfying an open condition \eqref{eq:cond-A}
    (to be described later)  on the curve:
    \[
  y^2 = x  (x - n)  (x - n (u^2 + 1)), \quad \text{where } n = m (m^2 + 1) (2u - m(u^2 - 1)).
    \]
\end{theorem}
%{\color{red}TODO: explain "almost".} \mcom{I'm not sure what you mean by almost in this Theorem, to be honest} \mcomanother{I also have a question about this. Once we have a solution with $m$ and $u$, we get infinitely many solutions by changing the $t$, don't we?}
This open condition \eqref{eq:cond-A} encodes the fact that the angles are positive and that the area is $<\pi$, and it also excludes two further points on the curve.
Computing the rank of the above curve (and proving it is $\geq 1$) we get:
%A computation of the rank of the above curve gives:
\begin{corollary}
    \label{cor:heron-angles-rank1}
    For almost all $m,u \in \Q$ with $m,u >0$ and $mu < 1$ there are infinitely many hyperbolic Heron triangles
    with rational area $m$ and one rational angle $u$.
\end{corollary}
Here ``almost'' means all values except possibly those lying on finitely many curves $\{f_i(m,u) = 0\}$ of 
with $\deg(f_i) \geq 8$.
Setting $u=1$ (which corresponds to the case of one right angle),  
it is easily verified that the $f_i(m,1)$ have no rational solution.
Thus the congruent number problem always has infinitely many solutions in the hyperbolic setting.

One can also parametrize hyperbolic Heron triangles using side lengths:
\begin{theorem}
    \label{th:heron-sides}
    There is a one-to-one correspondence between hyperbolic Heron triangles \break having two sides of lengths
    $\log(v)$ and $\log(w)$ with $v,w \in \Q$, and rational points satisfying an open 
    condition \eqref{eq:cond-B} (to be described later) on the curve:
    \[
        y^2 = x\left(x-(v-v^{-1})^2\right) \left(x-(w-w^{-1})^2\right).
    \]
\end{theorem}
Here the condition \eqref{eq:cond-B} describes the fact that the lengths are positive and satisfy the triangle inequality.
It turns out that (perhaps surprisingly) this curve has generically rank 0 over $\Q$, suggesting that it is 
harder to complete two sides to a rational triangle, than it is to find a triangle with a fixed angle and area.

In the Euclidean world, a further question one can ask is that of the rationality of the medians.
This was first asked (and solved) by Euler in the case of one rational median \cite{Euler}. 
The problem for Heron triangles having 3 rational medians is problem D21 in Guy's book \cite{Guy}; it is still open as of today.
The two-median problem was solved by Buchholz and Rathbun \cite{BR97, BR98}.

In the hyperbolic setting we have again two choices in our translation of median: the (hyperbolic) line 
from one vertex meeting the opposite edge in its midpoint, or the line from one vertex separating the 
triangle into two triangles of equal area.
We will call the first one the \emph{median} and the second one the \emph{area bisector}.
These two lines are not the same in general (one can easily be convinced by considering ideal 
triangles), but they coincide in an isosceles triangle (for the lines passing through the apex).

The first (negative) result  along those lines concerns the simple case of equilateral triangles.
\begin{propositionAbs} \label{prop:equilateral}
  There are no equilateral hyperbolic Heron triangles. Moreover, equilateral triangles with rational 
  sides or rational angles have no rational medians/area bisectors.
\end{propositionAbs}

With similar methods as those used for Theorem~\ref{th:heron-angles} and Theorem~\ref{th:heron-sides}, we 
can parametrize hyperbolic triangles with one rational median and Heron triangles with one rational 
area bisector using elliptic curves. 
However, these curves are quite complicated, so we will present them in Section~\ref{sec:medians} and 
\ref{sec:area-bisectors} respectively.
We will only give one corollary of this parametrization, in the case of medians:
\begin{theorem}
    For almost all values $u,w \in \Q$ there are infinitely many hyperbolic triangles having rational side 
    lengths, two of which given by $a = 2 \log(u)$ and $b = \log(w)$, and one rational median (intersecting 
    side $a$).
\end{theorem}
Here almost all means all but those cut out by a curve in $u,w$.

The paper is organized as follows.
Section~\ref{sec:heron-angles} and Section~\ref{sec:heron-sides} are dual to each other, and cover the 
parametrization of Heron triangles in terms of angles (Theorem~\ref{th:heron-angles}) and sides 
(Theorem~\ref{th:heron-sides}) respectively.
Section~\ref{sec:equilateral} is focused on medians and area bisectors in the simple case of equilateral 
triangles.
Finally, in Section~\ref{sec:medians} we give the parametrization of hyperbolic triangles with rational side 
lengths and one rational median, while Section~\ref{sec:area-bisectors} is devoted to the dual computation of 
the parametrization of Heron triangles with one rational area bisector.

%% file: heron-angles.tex
\subsection{Finding the angle parametrization}
In this section we give the parametrization of hyperbolic Heron triangles in terms of angles and area, and 
prove Theorem~\ref{th:heron-angles} and Corollary~\ref{cor:heron-angles-rank1}.
Note that throughout this paper we only consider the case of non-degenerate bounded triangles 
(i.e., with no vertex at infinity), so that its angles and area are always positive and its side lengths finite.

Let $\alpha, \beta, \gamma >0$ denote the angles of a hyperbolic triangle.
Assume they are rational (as defined in the introduction) i.e., 
that $e^{i\alpha}, e^{i\beta}, e^{i\gamma} \in \Q[i]$.
Since the area $A = \pi - \alpha - \beta - \gamma$ we get that $A$ is also rational: $e^{iA} \in \Q[i]$.

If we denote by $a$ (resp.\ $b,c$) the side length opposite $\alpha$ (resp.\ $\beta, \gamma$),
the hyperbolic law of cosines (for the angles) reads:
\begin{equation}
    \label{eq:law-cos-1}
 \sin(\beta) \sin(\gamma) \cosh(a) = \cos(\alpha) + \cos(\beta) \cos(\gamma)
\end{equation}
and similar formulas involving $\cosh(b), \cosh(c)$.
As the angles are rational, it follows that $\cosh(a), \cosh(b), \cosh(c) \in \Q$. 
To get a hyperbolic Heron triangle, it thus only remains to find the condition that $\sinh$ of $a,b,c$ are also 
rational.

The (inverse of) the hyperbolic law of sines gives
\[
  \frac{\sinh(a)}{\sin(\alpha)} = 
  \frac{\sinh(b)}{\sin(\beta)} = 
  \frac{\sinh(c)}{\sin(\gamma)}.
\]
Multiplying by $\sin(\alpha)\sin(\beta)\sin(\gamma)$ we get:
\[
  \sinh(a)\sin(\beta)\sin(\gamma) = 
  \sinh(b)\sin(\alpha)\sin(\gamma) = 
  \sinh(c)\sin(\alpha)\sin(\beta).
\]
Call this quantity $\Delta_1$.
Squaring \eqref{eq:law-cos-1} we get 
\[
    \sin(\beta)^2 \sin(\gamma)^2 (\sinh(a)^2 + 1) = (\cos(\alpha) + \cos(\beta) \cos(\gamma))^2
\]
and thus:
\begin{equation}\label{eq:Heroncondition}
  \Delta_1^2 = (\cos(\alpha) + \cos(\beta) \cos(\gamma))^2 - \sin(\beta)^2 \sin(\gamma)^2 \in \Q.
\end{equation}
Using trigonometric identities, we can rewrite this as a symmetric expression in 
$\alpha, \beta, \gamma$: 
\[
    \begin{array}{rcl}
        2\Delta_1^2 &=&  \cos(-\alpha + \beta + \gamma) + \cos(\alpha - \beta + \gamma) 
            + \cos(\alpha + \beta - \gamma) \\
                    && + \cos(\alpha + \beta + \gamma) + 
                \cos(2 \alpha) + \cos(2 \beta) + \cos(2 \gamma) + 1,
    \end{array}
\]
and we have that $\Delta_1 \in \Q$ if and only if all three of $\sinh(a)$, $\sinh(b)$, and $\sinh(c)$ are 
rational.
Substituting $\gamma = \pi - A - \alpha - \beta$ and expanding the cosines, we obtain
(writing $c_A = \cos(A), s_A = \sin(A)$, etc...):
%, we get 
%\[
    %\begin{array}{rcl}
        %2\Delta_1^2 &=& -\cos(A + 2\alpha) - \cos(A + 2\beta) - \cos(A+ 2\alpha + 2 \beta) 
                %\\ && - \cos(A) + \cos(2 \alpha) + \cos(2 \beta) + \cos(2 (A + \alpha + \beta)) + 1.
    %\end{array}
%\]
%Expanding all the cosines, we obtain (writing $c_A = \cos(A)$ etc...)
%expression correcte: (vérifié sur sage)
%val2 = (cos(A)^2 - sin(A)^2)*((cos(a)*cos(b) - sin(a)*sin(b))^2 - (cos(a)*sin(b) + cos(b)*sin(a))^2) \
%- 4*sin(A)*cos(A)*(cos(a)*sin(a)*(cos(b)^2 - sin(b)^2) + cos(b)*sin(b)*(cos(a)^2 - sin(a)^2)) \
%- cos(A)*( (cos(a)*cos(b) - sin(a)*sin(b))^2 - (cos(a)*sin(b) + cos(b)*sin(a))^2 + cos(a)^2 - sin(a)^2 + cos(b)^2 - sin(b)^2 + 1) \
%+ 4*sin(A)*(cos(a)*sin(a)*cos(b)^2 + cos(b)*sin(b)*cos(a)^2) \
%+ cos(a)^2 - sin(a)^2 + cos(b)^2 - sin(b)^2 + 1
\begin{equation}\label{eq:c_Aandfriends}
  \begin{array}{rcl}
      2 \Delta_1^2 &=& 
  (c_A^2 - s_A^2)\big[(c_\alpha c_\beta - s_\alpha s_\beta)^2 -  (c_\alpha s_\beta + c_\beta s_\alpha)^2\big]
                \\ && 
  - 4 c_A s_A\big[c_\alpha s_\alpha (c_\beta^2 - s_\beta^2) + c_\beta s_\beta(c_\alpha^2 - s_\alpha^2)\big]
                \\ && 
  - c_A\big[(c_\alpha c_\beta - s_\alpha s_\beta)^2 -  (c_\alpha s_\beta + c_\beta s_\alpha)^2
      + 2c_\alpha^2 + 2c_\beta^2 - 1\big]
                 \\&& +
  4 s_A(c_\alpha s_\alpha c_\beta^2 + c_\beta s_\beta c_\alpha^2) + 2 c_\alpha^2 + 2 c_\beta^2 - 1.
    \end{array}
\end{equation}

We wish to express this in terms of \emph{rational} angles.
Setting
\[  \renewcommand{\arraystretch}{1.5}
    \arraycolsep=1.5em
    \begin{array}{lll}
        \cos(A) = \frac{1 - m^2}{1+m^2}, & 
        \cos(\beta) = \frac{1-u^2}{1+u^2}, &
        \cos(\alpha) = \frac{1-t^2}{1+t^2},        \\
        \sin(A) = \frac{2m}{1+m^2}, &  
        \sin(\beta) = \frac{2u}{1+u^2}, &
        \sin(\alpha) = \frac{2t}{1+t^2}, 
  \end{array}
\]
and $w = (m^2+1)(u^2 + 1)(t^2 + 1)\Delta_1$,  equation \eqref{eq:c_Aandfriends} rewrites as:
\begin{align} \label{eq:w-t-angles}
    w^2  = & 4 m  (mu^2 - m - 2u)  (mt^2  -2t - m) \big[(mu^2 - m - 2u)t^2 
                 \\& + (-4mu - 2u^2 + 2)t - mu^2 + m + 2u\big]. \nonumber
\end{align}        
Finally, setting $n = m(m^2 + 1) (2u - m(u^2-1))$ and applying the change of variables 
{\small
\begin{align*}
    y =& -\frac{(2u - m (u^{2} -1) )}{4  t^{3}}
\big[2(2u - m (u^{2} -1) ) (m u - 1) (m + u) m^{2} 2(t^{2} - 3) 
        \\& +2  (m^{2} u^{2} - m^{2} - 6  m u - 2  u^{2} + 2) (2u - m (u^{2} -1)) m^{2} t^{2} + 2  (2u - m (u^{2} -1))^{2} m^{3} 
        \\&-  (m u - 1) (m + u) m t w + (2u - m (u^{2} -1)) m^{2} w\big]
\\
        x  = &\frac{(2u - m (u^{2} -1))}{4t^2} \big[-4   (m u - 1) (m + u) m t +2  (m^{2} u^{2} + m^{2} - 2  m u + 2) 
        m t^{2}  
        \\&+ 2  (2u - m (u^{2} -1)) m^{2} + m w\big],
    \end{align*}}
with inverse
{\small
\begin{align*}
    t  =& - \frac{m(2u - m(u^2-1))(x - (m+u)^2(m^2+1))}%
    {y + (m+u)(1 - mu)(x-n)} \\
   % w =& 2 m (2u - m(u^2-1))(y + (m+u)(1-mu)(x-n))^{-2}\\
      %& \times \big[2x^3 - y^2 + x^2 (m^2+1)(m^2u^4 + m^2u^2 - 2mu^3 - 5m^2 - 10mu - 3u^2) \\
      %& + 2 y (m+u)u^2(mu-1)(m^2+1)^2 -2 x m(m+u)^2(u^2+2)(m^2+1)^2(mu^2-m-2u) \\
     % & - m^2 (m+u)^2(u^2+1)(m^2+1)^3(mu^2 - m - 2u)^2\big]\\
     w=&  2 m (2u - m(u^2-1))(y + (m+u)(1-mu)(x-n))^{-2}\\
      & \times \big[x^3-3(m^2+1)(u+m)^2x^2 \\&+m(m^2+1)^2(2u - m(u^2-1))(m^2u^4+2u^4+2mu^3+2m^2u^2+4u^2+6mu+3m^2)x \\&-m^2(m^2+1)^3(u+m)^2(u^2+1)(2u - m(u^2-1))^2+
      2(m^2+1)^2u^2(u+m)(mu-1)y\big]\\
\end{align*}
}
%\mcom{I've just rechecked the above change and I agree with it.}
we get the equation: 
\begin{equation}\label{eq:angles}
    y^2 = x  \big(x - n\big)  \big(x - n (u^2 + 1)\big).
\end{equation}
Its discriminant is $2^4 u^4 (u^2 + 1)^2  n^6 = 2^4 u^4 m^6 (u^2 + 1)^2 (m^2 + 1)^6 (2u - m(u^2-1))^6$.

We are ready to complete the proof of Theorem~\ref{th:heron-angles}.
\begin{proof}[Proof of Theorem~\ref{th:heron-angles}]
    It just remains to exhibit the open condition (A), which encodes the fact that the parameters give rise to 
    an actual hyperbolic triangle.
    Assume first that the inverse of the change of variables is defined; this happens everywhere but on the 
    line $y + (m+u)(1 - mu)(x-n) =0$. 
    We claim that the conditions on $m,u,t$, along with their meaning are given by the following table:
\begin{table}[ht]
  \begin{tabular}{c|c}
    Condition & Translation\\\hline
  $m > 0$ & area $ A >0$ \\
  $t > 0$ & $\alpha > 0$ \\
  $u > 0$ & $\beta> 0$ \\
  $(tm + tu + mu -1)(tmu - t - m - u) > 0$ & $\gamma > 0$ \\
  $mu < 1$ & $\alpha + \beta + \gamma < \pi$.
    \end{tabular}
    \caption{Conditions for Theorem \ref{th:heron-angles} \label{table:conditionsheronangles}}
\end{table}
%\mcom{Here again, I wonder about the extreme cases, such as $m=0$. Do we need to say anything about them?}
The first three conditions follow trivially from the circle parametrization, and so does the fourth one if one writes 
the formula for the rational angle corresponding to $\gamma$.
For the fifth one, assuming the four others, we know that $0 < A, \alpha, \beta, \gamma < \pi$, and we must 
find a condition encoding $A + \alpha + \beta + \gamma = \pi$.
By construction, any solution to \eqref{eq:angles} gives rise to a set of parameters satisfying this last 
equation \emph{modulo} $2\pi$; thus we have $A + \alpha + \beta + \gamma \in \{\pi, 3\pi\}$ and we wish to
eliminate the $3\pi$ case.
For such a set of parameters, the $3\pi$ case happens exactly when any two elements in 
$A, \alpha, \beta, \gamma$ sum up to $>\pi$. 
Hence we can just pick the condition $A + \beta < \pi$, which translates as $mu < 1$ if one writes the 
formula for the rational angles.

Now if $m,u$ are given parameters with $m,u>0$ and $mu < 1$, the conditions from Table \ref{table:conditionsheronangles} can be simplified to the 
following open condition on the variable $t$:
\begin{equation}\label{eq:cond-A}
  \tag{A}  0 < t < \frac{1 - mu}{m+u}.
\end{equation}
Using the above change of variables, we get an open condition in terms of $x,y$ encoding the desired properties.

It remains to treat the case of points where the expression for $t$ is not defined, i.e., points of the curve 
\eqref{eq:angles} lying on the line $y + (m + u)(1 - mu)(x - n)=0$.
There are 3 such points: the torsion point $T = (n,0)$, the point 
\[
P =\big((m^2+1)(m+u)^2,u^2(m^2+1)^2(m+u)(mu-1)\big)
\]
and the third point $-(P+T)$.
Looking back at points on \eqref{eq:w-t-angles} mapping to the line, we find only one: the point $(w,t)$ with 
\[
    t = \frac{2m(m+u)(1-mu)(2u - m(u^2-1))}{m^4u^4 - m^4u^2 -4m^3u^3 +m^4 +4m^3u +6m^2u^2 + u^2}.
\]
Its image is the point $P$.
So whenever this point exists and fulfills the condition \eqref{eq:cond-A} we get a pre-image for $P$, and 
otherwise we do not. 
The condition \eqref{eq:cond-A} thus needs to be strengthened to exclude those points.
\end{proof}

\subsection{Rank computations} \label{sub:angles-rank}
In this section we compute the rank of the above curve (fixing one parameter $u$ or $m$) and prove 
Corollary~\ref{cor:heron-angles-rank1}.
We will write $E_{m,u}$ for the curve given by \eqref{eq:angles}, and $E_m$ (resp.\ $E_u$) for the same curve 
seen over $\C(u)$ (resp.\ $\C(m)$).

\begin{lemma}\label{lem:rank-computations}
    The ranks of the $K3$ surfaces $E_m(\C(u))$ and $E_u(\C(m))$ fulfill the inequalities:
    \[
     1 \leq \quad \rk(E_m(\C(u))),\rk(E_u(\C(m))) \quad \leq 2.
    \]
    Moreover, the point 
    \[
    P(m,u) =\big((m^2+1)(m+u)^2,u^2(m^2+1)^2(m+u)(mu-1)\big)
    \]  
    is a point of infinite order on $E_{m,u}$, and the torsion group is isomorphic to 
$\Z/2\Z\times\Z/2\Z$, with the points of order two given by 
$(0,0), (n,0),$ and $(n(u^2+1),0)$.
    \end{lemma}
\begin{proof}
    The lower bound follows from the fact that $P$ is on the curve $E_{m,u}$ and of infinite order, 
    which is easily verified by inspection. Alternatively one can see that the height pairing of the 
    Mordell--Weil group on $E_u$ (resp.\ on $E_m$) gives $h(P)=2$.

    For the upper bound, we apply Tate's algorithm \cite[IV.9]{Silverman}.
    For $E_u(\C(m))$ we see that the singularities $m = 0, \pm i, \frac{2u}{u^2-1}$ are all of type $I_0^*$
    in the Kodaira classification, and so the number of components of each singular fiber is 5.
    Now by the Shioda--Tate formula (see \cite[Corollary~1.5]{Shioda}, or alternatively 
    \cite[Corollary~6.7)]{SS-book}) we have
    \[
        \rho(E_u)=\rk(E_u(\C(m))) + 2 + 4 \cdot(5-1) = \rk(E_u(\C(m)))+18
    \]
    where $\rho(E_u)$ is the Picard number of $E_u$.
    Since $E_u$ is a $K3$ surface, we have ${\rho(E_u) \leq 20}$, and thus 
    $\rk(E_u(\C(m))) \leq 2$.

    For $E_m(\C(u))$, we use the same formula. 
    The singularities at $u=0$ and $u=\infty$ are of type $I_4$, the ones at $u=\pm i$ of type $I_2$ and the 
    ones at the roots of $m u^2 -2u -m$ of type $I_0^*$.
    Applying the Shioda--Tate formula again, 
    \[
    \rho(E_m)=\rk(E_m(\C(u))+2+2\cdot (4-1)+2\cdot (2-1)+2\cdot(5-1)=\rk(E_m(\C(u))+18.
    \]
    Thus $\rk(E_m(\C(u))) \leq 2$ also in this case.
    
  Finally, it is immediate to see that $(0,0), (n,0),$ and $(n(u^2+1),0)$ are points of order~2. Recall that the Euler characteristic satisfies $\chi=2$ for $K3$ surfaces. We combine this with the bound on the rank and Table (4.5) in \cite{MirandaPersson} to conclude that $E_u$ (resp. $E_m$) has torsion isomorphic to $\Z/4\Z\times \Z/2\Z$ or $\Z/2\Z\times \Z/2\Z$. Then one can check directly that the points of order 2 cannot be written as twice a point in  $E_u(\C(m))$ (resp. in $E_m(\C(u))$).  
\end{proof}

If we set $u=1$ in the case of $E_u$, the singularities are at $m = 0, \pm i, \infty$ and the previous Lemma still applies. In this particular case, we can identify a second point of infinite order given by 
\[
Q(m)=\big(2m(m+1)^2,i4m^2(m^2-1)\big).
\]
One can prove that the height pairing of the Mordell--Weil group on $E_1$ gives $h(Q)=2$, and that $\langle P,Q\rangle=0$, showing that $P(m,1)$ and $Q(m)$ are independent, and that the rank of $E_1(\C(m))$ is exactly 2.

We move on to the proof of Corollary~\ref{cor:heron-angles-rank1}.
\begin{proof}[Proof of Corollary~\ref{cor:heron-angles-rank1}]
    We need to show that for each fixed $m,u >0$ with $mu < 1$, there are infinitely many rational points
    on the curve $E_{m,u}$ satisfying the open condition \eqref{eq:cond-A}.
    By a theorem of Poincaré and Hurwitz (see \cite[Satz~11, p.~78]{Skolem})
    the rational points $E_{m,u}(\Q)$ of $E_{m,u}$ are dense in $E_{m,u}(\R)$ provided $E_{m,u}(\Q)$ is 
    infinite and both connected components of $E_{m,u}(\R)$ contain a rational point. 
    Hence the corollary will follow by density, once these two conditions are proven.

    Now by Lemma~\ref{lem:rank-computations}, $E_{m,u}(\Q)$ has positive rank, so it is 
    infinite.
    And since the 3 non-trivial torsion points are given by
    \[
      \big(0,0\big), \quad \big(n,0\big), \quad \big(n(u^2 + 1),0\big)
    \]
   and are rational, we conclude that both components have rational points.
\end{proof}

%% file: heron-sides.tex
\subsection{Finding the side lengths parametrization}
This section is devoted to the parametrization of hyperbolic Heron triangles using side lengths.
The arguments are quite similar to those of Section~\ref{sec:heron-angles}, but using the (dual) hyperbolic 
law of cosines for the side lengths.

Let $a,b,c$ denote the side lengths of a (non-degenerate bounded) hyperbolic triangle, and assume that 
$e^a, e^b, e^c$ are rational.
Let $\alpha$ (resp.\ $\beta, \gamma$) be the angles opposing the sides of length $a$ (resp.\ $b,c$).
By the law of cosines (for the side lengths)
\begin{equation}
    \label{eq:law-cos-2}
    \sinh(b) \sinh(c) \cos(\alpha) = \cosh(b)\cosh(c)- \cosh(a)
\end{equation}
and hence $\cos(\alpha)$ is rational (and thus also $\cos(\beta)$ and $\cos(\gamma)$ for similar reasons). 
%\mcom{This excludes degenerate cases where the $\sinh$ of a side is zero}

Multiplying the hyperbolic law of sines 
\[
  \frac{\sin(\alpha)} {\sinh(a)}= 
  \frac{\sin(\beta)}{\sinh(b)} = 
  \frac{\sin(\gamma)}{\sinh(c)}
\]
by $\sinh(a)\sinh(b)\sinh(c)$, we get 
\[
    \sin(\alpha) \sinh(b) \sinh(c) = \sin(\beta) \sinh(a) \sinh(c) = \sin(\gamma) \sinh(a) \sinh(b).
\]
Call this quantity $\Delta_2$. 
We have that $\Delta_2$ is rational if and only if $\sin(\alpha), \sin(\beta), \sin(\gamma)$ are rational.

As in Section~\ref{sec:heron-angles}, we square \eqref{eq:law-cos-2} to get
\[
    \sinh(b)^2 \sinh(c)^2(1 - \sin(\alpha)^2) = ( \cosh(b)\cosh(c)- \cosh(a))^2.
\]
Hence 
\[
    \Delta_2^2 = \sinh(b)^2 \sinh(c)^2 - (\cosh(b) \cosh(c)  - \cosh(a))^2 \in \Q.
\]  
As a side note, observe that if $A$ denotes the area of the triangle, we have
\[
    \sin(A) = -\sin(\alpha) \sin(\beta) \sin(\gamma) + \sin(\alpha) c_\beta c_\gamma + \sin(\beta) c_\alpha c_\gamma + \sin(\gamma) c_\alpha c_\beta
\]
where $c_\alpha = \cos(\alpha)$, etc...
Since $\frac{\Delta_2}{\sin(\alpha)} \in \Q$, and similarly for $\beta, \gamma$, we conclude that $\sin(A) = r\Delta_2$  for some $r \in \Q$. 
Thus a hyperbolic triangle with rational side lengths has rational area exactly when $\Delta_2 \in \Q$, i.e., when all its angles are rational.

Rewriting the equation for $\Delta_2^2$ in a symmetric way, we get:
\[
  \Delta_2^2 = 1 - \cosh(a)^2 - \cosh(b) ^2 - \cosh(c)^2 + 2 \cosh (a) \cosh (b) \cosh (c).
\]  
Letting $u = e^a, v = e^b$ and $w = e^c$ (so that $u,v,w \in \Q$), this equation rewrites as:
\[
    4 u^2 v^2 w^2 \Delta_2^2 = (uv - w) (uw - v) (vw - u) (uvw -1).
\]
We introduce the following change of variables
\[
y=\frac{2u(v^2-1)(w^2-1)(v^2w^2-1)\Delta_2}{vw(u-vw)^2}, \quad x=\frac{(v^2-1)(w^2-1)(uvw-1)}{vw(vw-u)},\\
\]
with inverse 
\begin{align*}
    \Delta_2 &=\frac{(v^2-1)(w^2-1)(v^2w^2-1)y}{2\big(x+(v^2-1)(w^2-1)\big)\big(v^2w^2x+(v^2-1)(w^2-1)\big)},\\
    u &=\frac{v^2w^2x+(v^2-1)(w^2-1)}{vw\big(x+(v^2-1)(w^2-1)\big)}.
\end{align*}

This leads to the desired equation:
\begin{equation}\label{eq:sides}
    y^2 = x\left(x-(v-v^{-1})^2\right) \left(x-(w-w^{-1})^2\right).
\end{equation}
Its discriminant is given by $2^4(v-v^{-1})^4(w-w^{-1})^4(v^{-1}w-vw^{-1})^2(vw-v^{-1}w^{-1})^2$. 

We can now proceed towards the proof of Theorem~\ref{th:heron-sides}.
\begin{proof}[Proof of Theorem~\ref{th:heron-sides}]
    It suffices to exhibit the open condition (B) that ensures the parameters give rise to a hyperbolic triangle.
    First, all the side lengths must be positive, whence $u >1$, $v>1$ and $w >1$.
    Second, the three triangle inequalities must be satisfied, i.e., $u < vw, v < uw$ and $w < uv$.
    Assuming $v,w$ are fixed and $>1$, we get the condition:
    \begin{equation} \label{eq:cond-B} \tag{B}
        \max\left(\frac{v}{w}, \frac{w}{v}\right) < u < vw.
    \end{equation}
    Under this condition, the above change of variable is always defined (and so is its inverse since $x$ is 
    necessarily positive).
    Thus the theorem is proven.
\end{proof}

\subsection{Rank Computations}%
\label{sub:sides-rank}
This section is similar to \ref{sub:angles-rank}.
Let $E_{v,w}$ denote the curve given by \eqref{eq:sides}, and let $E_v$ (resp.\ $E_w$) denote the same curve 
seen over $\C(w)$ (resp.\ $\C(v)$)).
Our goal is to give bounds on the rank of $E_v$ and $E_w$; since the equation for $E_{v,w}$ is symmetric in $v$ 
and $w$, it is enough to consider the curve $E_v$.
\begin{lemma}\label{lem:rank-computations-sides}
    The rank of the $K3$ surface $E_v(\C(w))$ satisfies
    \[
     1 \leq \quad \rk(E_v(\C(w)))\quad \leq 2.
    \]
    Moreover, the point 
    \[
      R(v,w) = \big(-vw(v-v^{-1})(w-w^{-1}),ivw(v-v^{-1})(w-w^{-1})(vw-v^{-1}w^{-1})\big)
    \]  
    is a point of infinite order on $E_{v,w}$.

    In addition, the torsion group is isomorphic to $\Z/4\Z\times\Z/2\Z$, generated by 
\[S_0(v,w)=\big((v-v^{-1})(w-w^{-1}),i(v-v^{-1})(w-w^{-1})(v^{-1}-w^{-1})(vw+1)\big)\]
and  \[S_1(v,w)=\big((v-v^{-1})^2, 0\big).\]

    \end{lemma}

\begin{proof}
The proof of this result is very similar to that of Lemma \ref{lem:rank-computations}. As before, 
the lower bound for the rank follows from the fact that $R$ is on the curve and of infinite order, 
    which is easily verified by inspection. More formally, one can see that the height pairing of the Mordell--Weil group on $E_v$  gives $h(R)=\frac{1}{2}$.
    
  We observe that $E_v$ has singularities at $w=0, \infty,\pm 1$ of type $I_4$ in the Kodaira classification, while the singularities at $w=\pm v, \pm v^{-1}$ are of type $I_2$. By the Shioda--Tate formula, we have
 \[ \rho(E_v)=\mathrm{rk}E_v(\C(w))+2+4\cdot(4-1)+4\cdot(2-1)=\mathrm{rk}E_v(\C(w))+18.\]
    Since $E_v$ is a $K3$ surface, we have ${\rho(E_v) \leq 20}$, and thus 
    $\rk(E_v(\C(w))) \leq 2$.

One can directly check that  $S_0$ and $S_1$ generate a subgroup isomorphic to $\Z/4\Z\times\Z/2\Z$. Combining the information about the rank and the Euler characteristic with  Table (4.5) in \cite{MirandaPersson} we conclude that the torsion group is given exactly by  $\Z/4\Z\times \Z/2\Z$. 
\end{proof}

Even if Lemmas \ref{lem:rank-computations} and \ref{lem:rank-computations-sides} are very similar from the point of view of the arithmetic of the involved $K3$ surfaces, they contain a fundamental difference for the geometric problem. In the case of  Lemma \ref{lem:rank-computations}, the point $P(m,u)$ is defined over $\Q(m,u)$ and generates a nontrivial solution to the question of the Heron hyperbolic triangle with given area and angle. In contrast, the point $R(v,w)$ of Lemma \ref{lem:rank-computations-sides} is certainly not defined over $\Q(v,w)$, and we speculate that there is no point of infinite order over $\Q(v,w)$. This results in the following: we do not know a priori whether there is a Heron hyperbolic triangle with given sides $v$ and $w$. This will depend on the choice of $v$ and $w$, much like the classical congruent number problem depends on the choice of the area.

%% file: equilateral.tex
This short section covers the specific case of (non-degenerate bounded) equilateral triangles.
\begin{proposition}
    There are no equilateral hyperbolic Heron triangles.
\end{proposition}
\begin{proof}
   Let $\alpha$ denote the angle of an equilateral hyperbolic triangle.
   The Heron condition \eqref{eq:Heroncondition} from Section~\ref{sec:heron-angles} in this case is:
   \[
       \Delta_1^2 = 2 \cos(\alpha)^3 + 3 \cos(\alpha)^2 -1 = (2\cos(\alpha)-1)(\cos(\alpha)+1)^2.
   \]
   Setting $u = \frac{\Delta_1}{\cos(\alpha)+1}$, this equation rewrites as $u^2 = 2 \cos(\alpha)-1$.
   Thus the solutions to the original equation are parametrized by 
   $\cos(\alpha) = \frac{u^2 + 1}{2}$ and $\Delta_1 = \frac{u^3 + 3u}{2}$, for $u \in \Q$.

   Now squaring the equation for $\cos(\alpha)$, and setting $v= 2\sin(\alpha)$, we get 
   \[
      v^2 = -u^4 -2u^2 +3. 
   \]
   Making the change of variables $u = \frac{x-1}{x+1},v = \frac{4y}{(x+1)^2}$ (with inverse 
   $x = \frac{1+u}{1-u}, y=\frac{v}{(u-1)^2}$), we get the Weierstrass form
   \[
      y^2 = x(x^2 + x + 1).
   \] 
   This has rank 0, and the only nontrivial torsion point is $(0,0)$ which does not give an actual triangle since $v=0$.
\end{proof}

We proceed to prove the second part of Proposition~\ref{prop:equilateral}:
\begin{proposition}
    If an equilateral hyperbolic triangle has either rational side lengths 
    or rational angles, then it has no rational median/area bisector.
\end{proposition}
\begin{proof}
  First observe that for equilateral triangles, mediators, bisectors, medians, and area bisectors all coincide,
  so we are free to use any property of these we want.
  Consider an equilateral hyperbolic triangle of side lengths $a$ and angles $\alpha$.
  Let $m$ denote the length of the median, and consider the half triangle defined by one median.
  This triangle has angles $\alpha, \frac{\alpha}{2}, \frac \pi 2$ and sides $a, \frac a 2, m$.

  First, assume the length $a$ is rational, i.e., that $e^a \in \Q$.
  By Pythagoras' theorem, $\cosh(m) \cosh(\frac a 2) = \cosh(a)$, so that $\cosh(m) \in \Q$ if and only 
  if $p = \cosh(\frac a 2) \in \Q$.
  Let $t = \sinh(m)$.
  Squaring Pythagoras' formula, we get the following equation for $t$:
  \[
      (1 + t^2) p^2 = (2 p^2 -1)^2 \quad \text{i.e.} \quad s^2 = 4 p^4 - 5 p^2 + 1, 
  \]
  writing $s = pt$.
  Changing variables 
 \[ s=\frac{9-x^2}{8x},\quad p=\frac{y}{4x}, \quad y=4p(-4s+8p^2-5),\quad x=-4s+8p^2-5.\]
  we get the following elliptic curve: 
  \[
    y^2 = x^3 + 10x^2 + 9x.
  \]
  The curve has rank 0, and the torsion is given by 
 \[
   E(\Q)_\mathrm{tors}=\langle (-3,6), (-1,0)\rangle \cong \Z/4\Z\times \Z/2\Z.
 \]
 Since we are looking for solutions with $p\not =0$, we need $y\not =0$. The only torsion points to consider are therefore $(-3,\pm 6)$ and $(3,\pm 12)$. However, we also need $t\not = 0$, leading to $s\not=0$ and $x\not = \pm3$. Therefore, there are no solutions. 
  
  Now assuming the angle $\alpha$ to be rational (i.e. $e^{i\alpha} \in \Q[i]$), the situation is entirely 
  similar.
  From the law of cosines (for the angles) we get: $\sin(\frac \alpha 2) \cosh(m) = \cos(\alpha)$, whence 
  $\cosh(m) \in \Q$ if and only if $p = \sin(\frac \alpha 2) \in \Q$.
  We set $t = \sinh(m)$ and square the equation to get the same equation as before:
  \[
    (1 + t^2) p^2 = (2 p^2 -1)^2.
  \]
  Thus also in this case, the median cannot be rational.
\end{proof}

%% file: medians.tex
The goal of this section is to study hyperbolic triangles with one rational median, in the same 
spirit as Euler's problem \cite{Euler}.
Consider a (non-degenerate bounded) hyperbolic triangle with sides $a,b,c$ having opposite angles $\alpha, \beta, \gamma$ (by abuse of 
notation we will let $a,b,c$ also denote the length of the sides).
Let $m$ denote the  median at angle $\alpha$, cutting side $a$ into two equal parts.
Denote by $\theta$ the angle at the intersection of $m$ and $a$, on the side of $\beta$; the one on the side of 
$\gamma$ is $\pi - \theta$.
Applying the cosine theorem in the two triangles, we get:
\begin{align*}
    \cosh(b) &= \cosh(m) \cosh(a/2) - \sinh(m) \sinh(a/2) \cos(\pi-\theta) \\
    \cosh(c) &= \cosh(m) \cosh(a/2) - \sinh(m) \sinh(a/2) \cos(\theta) 
\end{align*}
and thus 
\[  
    2\cosh(m) \cosh(a/2) = \cosh(b) + \cosh(c).
\]

Let us now assume that $a,b,c$ are rational side lengths, i.e., that $e^a, e^b, e^c \in \Q$.
In order for $\cosh(m)$ to be rational, it is thus necessary and sufficient that $\cosh(a/2)$ be rational.
Since $e^a \in \Q$, this is equivalent to $e^{a/2}\in \Q$.
For $\sinh(m) \in \Q$ we get the following condition from the above equation:
\[
    \left(\cosh(b) + \cosh(c)\right)^2 - 4 \cosh(a/2)^2 =4\sinh^2(m)\cosh^2(a/2)=\text{square}.
\]
Setting $u=e^{a/2}, v=e^{b}, w=e^{c}$, we need to solve
\[
(v^2w+w+vw^2+v)^2u^2-4v^2w^2(u^2+1)^2= t^2
\]
for $t,u,v,w \in \Q$.
Now applying the change of variables 
\[
    x  = 2w \big[ u^2w v^2 + u^2 (w^2 + 1) v -2u^4w - 3u^2w + ut - 2w \big]
\]
\begin{align*}
    y = 2uw \big [ &2u^2 w^2 v^3 + 3 w  u^2 (w^2 + 1) v^2  + 
      (-4u^4w^2 + u^2w^4 - 4u^2w^2 + 2uwt + u^2 - 4w^2) v 
  \\ &+  u (w^2 + 1) (uw + t) \big]
\end{align*}
with inverse
\[v=-\frac{1}{2xuw} \big[4  u  w^2  (w^2 + 1)  (u^2 + 1)^2 + u(w^2 + 1)x - y \big]\]
%{4u^5w^4+8u^3w^4+4uw^4+4u^5w^2+8u^3w^2+xuw^2+4uw^2+xu-y}{2xuw}\]
\begin{align*}
    t = -\frac{1}{4x^2uw} \big[ & - x^3 + 8 w^2 (u^2 + 1)^2 (2w^2(u^2+1)^2+(w^4+1)u^2) x \\
      &- 8  u  w^2  (w^2 + 1) (u^2 + 1)^2 y + 32  u^2  w^4  (w^2 + 1)^2  (u^2 + 1)^4\big],
\end{align*}
we get the equation
\begin{align} \label{eq:rat-median}
    y^2 & =x^3+(u^2w^4+2(4u^4+7u^2+4)w^2+u^2)x^2 \\
 \nonumber       & +8(u^2+1)^2w^2(u^2w^4+2(u^2+1)^2w^2+u^2)x +16u^2w^4(u^2+1)^4(w^2+1)^2.
\end{align}

If, similarly as previously, we let $E_{u,w}$ denote the elliptic curve given by \eqref{eq:rat-median} seen 
over $\Q$, where $u,w \in \Q$ are parameters, we get:
\begin{theorem}
    A hyperbolic triangle with rational sides $a = 2 \log(u)$, $b = \log(w)$ has a rational median 
    (intersecting side $a$) if and only if it corresponds (using the above change of variables)
    to a rational point on the elliptic curve $E_{u,w}$.
\end{theorem}

Let $E_u$ denote the curve $E_{u,w}$ seen over $\C(u)$.
As previously, we have the following lemma, a weaker analogue to 
Lemmas~\ref{lem:rank-computations} and \ref{lem:rank-computations-sides}.
\begin{lemma} \label{lem:rank-medians}
    The rank of the $K3$ surface $E_u$ satisfies 
    \[
      1 \leq \rk(E_u(\C(w))) \leq 4.
    \]
    Moreover the point 
    \[
        P(u,w) = \big(0,4u(u^2+1)^2w^2(w^2+1)\big)
    \]
    is a point of infinite order on the curve.
\end{lemma}
\begin{proof}
   The proof is entirely similar to that of  
   Lemmas~\ref{lem:rank-computations} and \ref{lem:rank-computations-sides}.
   The discriminant of $E_{u,w}$ is given by:
\begin{align*}
2^{12} u^4 w^8(u^2+1)^4(uw^2 + u -2(u^2+u+1)w )(uw^2 + u -2(u^2-u+1)w )\\
\times(uw^2 + u +2(u^2+u+1)w)(uw^2 + u +2(u^2-u+1)w).
\end{align*}
Looking at the Kodaira classification, we observe that $E_u$ has singularities of type $I_8$ at $w=0, \infty$, 
and of type $I_1$ for all the 8 others.
Thus the Shioda--Tate formula gives
\[
\rho(E_u)=\rk(E_u(\C(w)))+2+2\cdot (8-1)+8\cdot(1-1)=\rk(E_u(\C(w)))+16.
\]
Since $E_u$ is a $K3$ surface, we have $\rho(E_u) \leq 20$, and thus $\rk(E_u(\C(w))) \leq 4$.

The lower bound now follows from the face that $P(u,w)$ is of infinite order, which can be verified by 
direct computation.
%\mcom{The analysis over $E_w$ only gets a bound of 6 for the rank...}
\end{proof}

%% file: area-bisectors.tex
This section is similar to Section~\ref{sec:medians}, but focussing on hyperbolic triangles with rational 
area bisectors, instead of medians. 

Consider a hyperbolic triangle with sides $a,b,c$ having opposite angles 
$\alpha, \beta, \gamma$.
Let $m$ denote the  area bisector at angle $\alpha$, cutting $\alpha$ into $\alpha_1$ and $\alpha-\alpha_1$.
Denote by $\theta$ the angle at the intersection of $m$ and $a$, on the side of $\alpha_1$, and (assume) on the 
side of $\beta$.
Thus we have two triangles: one with angles $\alpha_1, \beta, \theta$ and one with 
$\alpha - \alpha_1, \gamma, \pi-\theta$.

By the law of cosines (for the angles) we have
\[
  \sin(\alpha_1) \sin(\beta)\cosh(c) = \cos(\theta) + \cos(\alpha_1)\cos(\beta).
\]  
Combining this with the definition of area bisector 
\[
  2(\pi - \alpha_1 - \theta - \beta) = A \quad \text{i.e.}\quad
  \theta = \pi - \frac{A}{2} - \alpha_1 - \beta,
\]  
we get 
\begin{equation}\label{eq:acosthm}
  \sin(\alpha_1) \sin(\beta)\cosh(c) = - \cos\left(\frac{A}{2} + \alpha_1 + \beta\right) + 
  \cos(\alpha_1)\cos(\beta).
\end{equation}
Using trigonometric formulas, this rewrites as 
\begin{align*}
    \sin(\alpha_1)^{-2}& =1+\frac{1}{\tan(\alpha_1)^2} \\
                       &=\frac{2+\sin(\beta)^2\sinh(c)^2-2\cos\left(\frac{A}{2}\right)+2(1-\cosh(c))\sin(\beta)\sin\left(\frac{A}{2}+\beta\right)}{\left(\cos(\beta)-\cos\left(\frac{A}{2}+\beta\right)\right)^2}.
\end{align*}
Now using the law of cosines again:
\[
    \sin(\alpha) \sin(\beta)\cosh(c)  = \cos(\gamma) + \cos(\alpha)\cos(\beta)
\]
and setting $w_1 =\left(\cos(\beta)-\cos\left(\frac{A}{2}+\beta\right)\right)(\sin(\alpha_1))^{-1}\sin(\alpha)$, 
we get the equation
%\begin{align*}
%w_1^2 =&\sin(\alpha)^2+\cos(\beta)^2\sin(\alpha)^2+(\cos(\alpha)\cos(\beta)\cos(A)-\sin(\alpha)\sin(\beta)\cos(A)\\&-\sin(\alpha)\cos(\beta)\sin(A)-\cos(\alpha)\sin(\beta)\sin(A)-\cos(\alpha)\cos(\beta))^2\\&-2\cos(\beta)\sin(\alpha)^2\left(\cos\left(\frac{A}{2}\right)\cos(\beta)
%-\sin\left(\frac{A}{2}\right)\sin(\beta)
%\right)\\&-2\sin(\alpha)(-(\cos(\alpha)\cos(\beta)\cos(A)-\sin(\alpha)\sin(\beta)\cos(A)\\&-\sin(\alpha)\cos(\beta)\sin(A)-\cos(\alpha)\sin(\beta)\sin(A))+\cos(\alpha)\cos(\beta))\\&\times \left(\sin\left(\frac{A}{2}\right)\cos(\beta)+\cos\left(\frac{A}{2}\right)\sin(\beta)\right).
%\end{align*}
%\mcom{Do we try to write things nicely like in the other section?
\begin{align*}
w_1^2 =&s_\alpha^2+c_\beta^2s_\alpha^2+(c_\alpha c_\beta c_A-s_\alpha s_\beta c_A-s_\alpha c_\beta s_A-c_\alpha s_\beta s_A-c_\alpha c_\beta)^2\\
&-2c_\beta s_\alpha^2\left(\cos\left(\frac{A}{2}\right)c_\beta -\sin\left(\frac{A}{2}\right)s_\beta\right)\\
&-2s_\alpha(-(c_\alpha c_\beta c_A-s_\alpha s_\beta c_A-s_\alpha c_\beta s_A-c_\alpha s_\beta s_A)+c_\alpha c_\beta)\\
&\times \left(\sin\left(\frac{A}{2}\right)c_\beta+\cos\left(\frac{A}{2}\right)s_\beta\right),
\end{align*}
%}
where, as usual, $s_\alpha = \sin(\alpha)$, etc...

Assume now that our triangle has rational angles as well as rational \emph{half-area}.
We apply a similar change of variables as in Section~\ref{sec:heron-angles}, namely:
\[
  \renewcommand{\arraystretch}{1.5}
    \arraycolsep=1.5em
    \begin{array}{lll}
        \cos(\frac{A}{2}) = \frac{1 - n^2}{1+n^2}, & 
        \cos(\beta) = \frac{1-u^2}{1+u^2}, &
        \cos(\alpha) = \frac{1-t^2}{1+t^2}, 
        \\
        \sin(\frac{A}{2}) = \frac{2n}{1+n^2}, &  
        \sin(\beta) = \frac{2u}{1+u^2}, &
        \sin(\alpha) = \frac{2t}{1+t^2}.
  \end{array}
\]
Setting $w=\frac{w_1(n^2+1)^2(t^2+1)(u^2+1)}{4n}$, and clearing squares, we obtain:
{\footnotesize
\begin{align*}
    w^2 =& 4  (n + u)^2  (nu - 1)^2 t^4 + 4  (n + u)  (nu - 1)  (2n^3u + 3n^2u^2 - 3n^2 - 6nu - u^2 + 1) t^3 \\
        & + (n^6u^4 + 2n^6u^2 + 8n^5u^3 + 11n^4u^4 + n^6 - 8n^5u - 50n^4u^2 - 64n^3u^3 \\
        &- 13n^2u^4 + 11n^4 + 64n^3u + 86n^2u^2 + 24nu^3 + u^4 - 13n^2 - 24nu - 6u^2 + 1) t^2 \\
        &-4  (n + u)  (nu - 1)  (2n^3u + 3n^2u^2 - 3n^2 - 6nu - u^2 + 1) t + 4  (n + u)^2  (nu - 1)^2.
\end{align*}
}

We make the following final change of variables:
{\footnotesize
\begin{align*}
    y = &\frac{4(n u - 1) (n + u) }{t^3} \big[  2  (2  n^{3} u + 3  n^{2} u^{2} - 3  n^{2} - 6  n u - u^{2} + 1) (n u - 1) (n + u) t^{3}  \\
        &  + (n^{6} u^{4} + 2  n^{6} u^{2} + 8  n^{5} u^{3} + 11  n^{4} u^{4} + n^{6} - 8  n^{5} u - 50  n^{4} u^{2} -  64  n^{3} u^{3} - 13  n^{2} u^{4} \\
        &+ 11  n^{4} + 64  n^{3} u + 86  n^{2} u^{2} + 24  n u^{3} + u^{4} - 13  n^{2} - 24  n u - 6  u^{2} + 1) t^{2} \\
         &+ 8  (n u - 1)^{2} (n + u)^{2} - 6  (2  n^{3} u + 3  n^{2} u^{2} - 3  n^{2} - 6  n u - u^{2} + 1)(n u - 1) (n + u) t\\
        &  + 4  (n u - 1) (n + u) w - (2  n^{3} u + 3  n^{2} u^{2} - 3  n^{2} - 6  n u - u^{2} + 1)  t w\big],
  %     y = &\frac{4}{t^3} \big[  2  (2  n^{3} u + 3  n^{2} u^{2} - 3  n^{2} - 6  n u - u^{2} + 1) (n u - 1)^{2} (n + u)^{2} t^{3}  \\
    %    &  + (n^{6} u^{4} + 2  n^{6} u^{2} + 8  n^{5} u^{3} + 11  n^{4} u^{4} + n^{6} - 8  n^{5} u - 50  n^{4} u^{2} -  64  n^{3} u^{3} - 13  n^{2} u^{4} \\
     %   &+ 11  n^{4} + 64  n^{3} u + 86  n^{2} u^{2} + 24  n u^{3} + u^{4} - 13  n^{2} - 24  n u - 6  u^{2} + 1)(n u - 1) (n + u) t^{2} \\
       %  &+ 8  (n u - 1)^{3} (n + u)^{3} - 6  (2  n^{3} u + 3  n^{2} u^{2} - 3  n^{2} - 6  n u - u^{2} + 1)(n u - 1)^{2} (n + u)^{2} t\\
       % &  + 4  (n u - 1)^{2} (n + u)^{2} w - (2  n^{3} u + 3  n^{2} u^{2} - 3  n^{2} - 6  n u - u^{2} + 1) (n u - 1) (n + u) t w\big],
\end{align*}
\begin{align*}
    x = & \frac{1}{t^2}\big[8  (n u - 1)^{2} (n + u)^{2} - 4  (2  n^{3} u + 3  n^{2} u^{2} - 3  n^{2} - 6  n u - u^{2} + 1) (n u - 1) (n + u) t\\
        & +(n^{6} u^{4} - 2  n^{6} u^{2} - 4  n^{5} u^{3} + 2  n^{4} u^{4} + n^{6} + 4  n^{5} u - 8  n^{4} u^{2} - 24  n^{3} u^{3} - 7  n^{2} u^{4} + 2  n^{4} + 24  n^{3} u \\
        & +38  n^{2} u^{2} + 12  n u^{3} - 7  n^{2} - 12  n u - 4  u^{2}) t^{2} + 4  (n u - 1) (n + u) w\big],
\end{align*}
}
with inverse
{\footnotesize
\begin{align*}
    t = &-\big[4  (n u - 1) x (n + u)\big]\cdot \big[(2  n^{3} u + 3  n^{2} u^{2} - 3  n^{2} - 6  n u - u^{2} + 1)  (n u^{2} - n - 2  u)^{2} (n^{2} + 1)^{2}\\
 & - (2  n^{3} u + 3  n^{2} u^{2} - 3  n^{2} - 6  n u - u^{2} + 1) x - y\big]^{-1},
\end{align*} 
\begin{align*}
w =&- 2(n u - 1) (n + u) \big[ (2  n^{3} u + 3  n^{2} u^{2} - 3  n^{2} - 6  n u - u^{2} + 1)^{2} (n u^{2} - n - 2  u)^{4} 
    (n^{2} + 1)^{4} \\
   &-2  (2  n^{3} u + 3  n^{2} u^{2} - 3  n^{2} - 6  n u - u^{2} + 1) (n u^{2} - n - 2  u)^{2} (n^{2} + 1)^{2}  y  \\
   & +(2  n^{4} u^{4} - 8  n^{4} u^{2} - 20  n^{3} u^{3} - 7  n^{2} u^{4} + 2  n^{4} + 20  n^{3} u + 34  n^{2} u^{2} + 12  n u^{3} - u^{4} - 7  n^{2}  \\
   &- 12  n u - 6  u^{2} - 1) (n^{2} + 1) x^{2} - 2   x^{3}  +  y^{2} \big]\\
   &\times  \big[(2  n^{3} u + 3  n^{2} u^{2} - 3  n^{2} - 6  n u - u^{2} + 1)^{2} (n u^{2} - n - 2  u)^{4} (n^{2} + 1)^{4}
\\ & -2  (2  n^{3} u + 3  n^{2} u^{2} - 3  n^{2} - 6  n u - u^{2} + 1)^{2} (n u^{2} - n - 2  u)^{2}  (n^{2} + 1)^{2} x
\\ & - 2  (2  n^{3} u + 3  n^{2} u^{2} - 3  n^{2} - 6  n u - u^{2} + 1) (n u^{2} - n - 2  u)^{2} (n^{2} + 1)^{2} y 
\\ & + (2  n^{3} u + 3  n^{2} u^{2} - 3  n^{2} - 6  n u - u^{2} + 1)^{2} x^{2} \\
   & +2  (2  n^{3} u + 3  n^{2} u^{2} - 3  n^{2} - 6  n u - u^{2} + 1) y x + y^{2}\big]^{-1},
%w =&-\big[2  ((2  n^{3} u + 3  n^{2} u^{2} - 3  n^{2} - 6  n u - u^{2} + 1)^{2} (n u^{2} - n - 2  u)^{4} 
 %   (n^{2} + 1)^{4} (n u - 1) (n + u) \\
 %  &-2  (2  n^{3} u + 3  n^{2} u^{2} - 3  n^{2} - 6  n u - u^{2} + 1) (n u^{2} - n - 2  u)^{2} (n^{2} + 1)^{2} (n u - 1) y (n + u)  \\
 %  & +(2  n^{4} u^{4} - 8  n^{4} u^{2} - 20  n^{3} u^{3} - 7  n^{2} u^{4} + 2  n^{4} + 20  n^{3} u + 34  n^{2} u^{2} + 12  n u^{3} - u^{4} - 7  n^{2}  \\
 %  &- 12  n u - 6  u^{2} - 1) (n^{2} + 1) (n u - 1) x^{2} (n + u) - 2  (n u - 1) x^{3} (n + u) + (n u - 1) y^{2} (n + u))\big]\\
 %  &\times  \big[(2  n^{3} u + 3  n^{2} u^{2} - 3  n^{2} - 6  n u - u^{2} + 1)^{2} (n u^{2} - n - 2  u)^{4} (n^{2} + 1)^{4}\\ & -2  (2  n^{3} u + 3  n^{2} u^{2} - 3  n^{2} - 6  n u - u^{2} + 1)^{2} (n u^{2} - n - 2  u)^{2}  (n^{2} + 1)^{2} x\\ & - 2  (2  n^{3} u + 3  n^{2} u^{2} - 3  n^{2} - 6  n u - u^{2} + 1) (n u^{2} - n - 2  u)^{2} (n^{2} + 1)^{2} y \\ & + (2  n^{3} u + 3  n^{2} u^{2} - 3  n^{2} - 6  n u - u^{2} + 1)^{2} x^{2} \\   & +2  (2  n^{3} u + 3  n^{2} u^{2} - 3  n^{2} - 6  n u - u^{2} + 1) y x + y^{2}\big]^{-1},
\end{align*}
}
and we get 
\begin{align}
    \label{eq:area-bisector}
 y^2=&(x-(n^2+1)^2(nu^2-2u-n)^2)(x^2-(n^2+1)(n^4u^4-8n^2u^4-u^4-16n^3u^3\\
 \nonumber   & +16nu^3-6n^4u^2+32n^2u^2- 10u^2+16n^3u-16nu+n^4-8n^2-1)x\\
  \nonumber     & -(n^2+1)^2(nu^2-2u-n)^2(3n^2u^2-u^2+2n^3u-6nu-3n^2+1)^2).
 \end{align}
 
 Let $E_{n,u}$ denote this elliptic curve \eqref{eq:area-bisector}, where $n,u$ are parameters.
 The data it encodes is the following.
 By assumption, $\frac{A}{2},\alpha,\beta,\gamma$ are all rational.
 Moreover, as in Section~\ref{sec:heron-angles}, it follows easily from the law of cosines that $\cosh(a)$, 
 $\cosh(b)$ and $\cosh(c)$ are also rational.
 Now by construction, a rational solution to \eqref{eq:area-bisector} corresponds to a triangle with 
 $\sin(\alpha_1)$ rational.
 In addition, \eqref{eq:acosthm} implies that $\cos(\alpha_1)\in \Q$, and thus $\alpha_1$ is rational.
 Since the area of the small triangle with angles $\alpha_1, \beta, \theta$ is rational (it is $\frac{A}{2}$),
 it follows that $\theta$ is also a rational angle.

 What is not encoded by the curve $E_{n,u}$ is the rationality of $\sinh(m)$.
 It is easy to see that this is actually equivalent to the original triangle being Heron: since all the 
 angles and areas under consideration are rational, we have $\sinh(m) \in \Q$ if and only if the small 
 triangle (with angles $\alpha_1, \beta, \theta$) is Heron (as explained in Section~\ref{sec:heron-angles}).
 Since $c$ is a side of this triangle, this happens exactly when $\sinh(c) \in \Q$, which, in turn, is true 
 if and only if the original triangle is Heron.

 Therefore, we have proven: 
 \begin{theorem}
     A hyperbolic Heron triangle has one rational area bisector if and only if it corresponds (using the 
     above change of variables) to a rational point of $E_{n,u}$.
 \end{theorem}

  This curve is more complicated than the ones of Section~\ref{sec:heron-angles} and \ref{sec:heron-sides}.
  Yet we have the following lemma, analog to Lemma~\ref{lem:rank-medians}.
  As before, we let $E_n$ denote the curve $E_{n,u}$ seen over $\C(u)$.
  \begin{lemma} \label{lem:rank-computations-area-bisectors}
    The rank of the surface $E_n$ satisfies 
    \[
        1 \leq \rk(E_n(\C(u))) \leq 4.
    \]
    Moreover, the point 
    \[
    Q(n,u)=\Big( 0,(n^2+1)^2(nu^2-2u-n)^2(3n^2u^2-u^2+2n^3u-6nu-3n^2+1)\Big)
    \]
    is a point of infinite order on the surface $E_n$, 
    and its torsion is either $\Z/2\Z, \Z/4\Z,$ $\Z/2\Z\times \Z/2\Z,$ or $\Z/4\Z\times \Z/2\Z$.
  \end{lemma}
  \begin{proof}
      It is not hard to see that $E_n$ is a $K3$ surface.
      Its discriminant is 
\begin{align*}
&2^{12}(n^2+1)^8(u+n)^4(nu-1)^4(u^2+1)^2((u^2-1)n-2u)^4\\
&\times \Big((n^4 + 18n^2 + 1)u^4+16n(n^2 - 3)u^3\\&+
2(n^4 - 30n^2 + 17)u^2-16n(n^2-3)u+n^4 + 18n^2 + 1\Big)
\end{align*}
We have singularities at $u=-n$, $\frac{1}{u}$, and the roots of $nu^2-2u-n$ of type $I_4$,  $\pm i$ of type $I_2$, and the roots of the last factor of type $I_1$. By the Shioda--Tate formula, 
\[\rho(E_n)=\mathrm{rk}(E_n(\C(u)))+2+4\cdot (4-1)+2\cdot(2-1)+4\cdot(1-1)=\mathrm{rk}(E_n(\C(u)))+16.\]
Thus the rank is $\leq 4$. 

The lower bound now follows from the fact that $Q(n,u)$ is of infinite order, which can be verified by 
direct computation.

Finally, it is immediate to see that $\left((n^2+1)^2(nu^2-2u-n)^2,0\right)$ is a point of order 2. 
We can conclude that the torsion is either $\Z/2\Z, \Z/4\Z, \Z/2\Z\times \Z/2\Z,$ or \linebreak 
$\Z/4\Z\times \Z/2\Z$. 
  \end{proof}